\documentclass[12pt,reqno]{amsart}

\usepackage{color}
\usepackage{amsmath, amsthm, amssymb}
\usepackage{amsfonts}
\usepackage[ansinew]{inputenc}
\usepackage[dvips]{epsfig}
\usepackage{graphicx}
\usepackage[english]{babel}
\usepackage{enumerate}
\usepackage{hyperref}
\theoremstyle{plain}
\newtheorem{thm}{Theorem}[section]
\newtheorem{cor}[thm]{Corollary}
\newtheorem{lem}[thm]{Lemma}
\newtheorem{prop}[thm]{Proposition}

\theoremstyle{definition}
\newtheorem{defi}[thm]{Definition}

\theoremstyle{remark}
\newtheorem{rem}[thm]{Remark}

\numberwithin{equation}{section}

\newcommand{\de}{\partial}

\newcommand{\fls}{(-\Delta)^s}
\newcommand{\R}{\mathbb{R}}

\newcommand{\N}{\mathbb{N}}

\newcommand{\average}{{\mathchoice {\kern1ex\vcenter{\hrule height.4pt
width 6pt depth0pt} \kern-9.7pt} {\kern1ex\vcenter{\hrule
height.4pt width 4.3pt depth0pt} \kern-7pt} {} {} }}

\def\R{\mathbb{R}}

\begin{document}

\title[Boundary regularity for nonlocal parabolic equations]{Higher-order boundary regularity estimates for nonlocal parabolic equations}

\author{Xavier Ros-Oton}
\address{Universit\"at Z\"urich, Institut f\"ur Mathematik, Winterthurerstrasse, 8057 Z\"urich, Switzerland}
\email{xavier.ros-oton@math.uzh.ch}

\author{Hern\'an Vivas}
\address{Departamento de Matem\'atica, FCEyN, Universidad Nacional de Mar del Plata, Conicet, Dean Funes 3350, 7600 Mar del Plata, Argentina}
\address{University of Texas at Austin, Department of Mathematics, 2515 Speedway, TX 78712 Austin, USA}
\email{havivas@mdp.edu.ar}


\begin{abstract}
We establish sharp higher-order H\"older regularity estimates up to the boundary for solutions to equations of the form $\partial_tu-Lu=f(t,x)$ in $I\times\Omega$ where $I\subset\R$, $\Omega\subset\R^n$ and $f$ is H\"older continuous.
The nonlocal operators $L$ that we consider are those arising in stochastic processes with jumps, such as the fractional Laplacian $(-\Delta)^s$, $s\in(0,1)$.

Our main result establishes that, if $f$ is $C^\gamma$ is space and $C^{\gamma/2s}$ in time, and $\Omega$ is a $C^{2,\gamma}$ domain, then $u/d^s$ is $C^{s+\gamma}$ up to the boundary in space and
$u$ is $C^{1+\gamma/2s}$ up the boundary in time, where $d$ is the distance to $\partial\Omega$.
This is the first higher order boundary regularity estimate for nonlocal parabolic equations, and is new even for the fractional Laplacian in $C^\infty$ domains.
\end{abstract}

\maketitle

\section{Introduction}\label{sec1}

In this paper we address the boundary regularity for solutions of nonlocal parabolic equations of the form
\begin{equation}
\label{eq.maineq}
\partial_tu - L u = f(t, x).	
\end{equation}
Here, $L$ is a nonlocal operator of the form
\begin{equation*}\label{eq.L0}
L u(t, x) = \textrm{P.V.}\int_{\R^n} \bigl(u(t, x+y)-u(t, x)\bigr)K(y)dy
\end{equation*}
with $K(y)=K(-y)$,
\begin{equation*}
0<\frac{\lambda}{|y|^{n+2s}}\leq K(y)\leq\frac{\Lambda}{|y|^{n+2s}}\quad\textrm{ and }\quad K(y)\textrm{ homogeneous.}
\end{equation*}
Equivalently, $L$ can be written as
\begin{equation}
\label{eq.L}
L u(t, x) = \int_{\R^n} \bigl(u(t, x+y)+u(t, x-y)-2u(t, x)\bigr) \frac{a(y/|y|)}{|y|^{n+2s}}dy,
\end{equation}
with $s\in (0,1)$ and $a:S^{n-1}\longrightarrow [\lambda,\Lambda]$, for some {ellipticity constants} $0<\lambda\leq\Lambda$.
In fact, to prove our higher-order regularity estimates we will also require the kernels to be regular, see Theorem \ref{thm.main} below.

Operators of this form arise naturally in the study of L\'evy processes with jumps, where the infinitesimal generators of stable processes take the form \eqref{eq.L}.
These operators have been studied both
from the point of view of Probability and Analysis and have become quite popular to model very different phenomena in Physics, Finance, Image processing, or Ecology; see \cite{Ap04,CT16,MK00,ST94} and references therein.
The most canonical and important example of such operators is given by the case where $a\equiv \textrm{ctt}$ on $S^{n-1}$ and $L$ becomes a multiple of
the fractional Laplacian $-\fls$:
\[L u(x) = c_{n,s}\int_{\R^n} \frac{u(x+y)+u(x-y)-2u(x)}{|y|^{n+2s}}dy,\]
see \cite{BV16,AV17,Ga17,L72,Go16} and references therein.

The regularity theory for integro-differential \emph{elliptic} equations started already in the fifties and sixties, and has experienced a huge development in the last 20 years; see for example \cite{Ro16} and references therein.
In particular, an important question in this context is to understand the boundary regularity of solutions, which presents many differences and challenges with respect to the case of local second order equations.
Such question has been studied in the works \cite{Bo97,BKK08,BKK15,GG15,GG14,RS1,RS2,SW99}, and it is now quite well understood.
For Dirichlet problems of the form
\begin{equation}\label{pb}
\left\{ \begin{array}{rcll}
L u &=&f&\textrm{in }\Omega \\
u&=&0&\textrm{in }\R^n\backslash\Omega,
\end{array}\right.
\end{equation}
with $L$ as above, the main known results of \cite{GG15,GG14,RS1,RS2} establish that, if $d(x)$ is the distance to $\partial\Omega$:
\begin{itemize}
\item[(a)] If $\Omega$ is $C^{1,1}$, then
\[f\in L^\infty(\Omega)\quad \Longrightarrow\quad u/d^s\in C^{s-\varepsilon}(\overline\Omega)\qquad\quad \textrm{for all}\ \varepsilon>0,\]
\item[(b)] If $\Omega$ is $C^{2,\gamma}$ and $a\in C^{1,\gamma}(S^{n-1})$, then
\[\qquad f\in C^\gamma(\overline\Omega)\quad \Longrightarrow\quad u/d^s\in C^{\gamma+s}(\overline\Omega)\qquad \textrm{for}\ \gamma\in(0,s),\]
whenever $\gamma+s$ is not an integer.
\item[(c)] If $\Omega$ is $C^\infty$ and $a\in C^\infty(S^{n-1})$, then
\[\quad f\in C^\gamma(\overline\Omega)\quad \Longrightarrow\quad u/d^s\in C^{\gamma+s}(\overline\Omega)\qquad \textrm{for all}\ \gamma>0,\]
whenever $\alpha+s\notin\mathbb Z$.
In particular, $u/d^s\in C^\infty(\overline\Omega)$ whenever $f\in C^\infty(\overline\Omega)$.
\end{itemize}
Part (a) was proved in \cite{RS1} (for any $a\in L^1(S^{n-1})$); (b) was established in \cite{RS2} in the more general context of fully nonlinear equations; and (c) was established in \cite{GG15,GG14} for elliptic pseudodifferential operators satisfying the $s$-transmission property.
Furthermore, when $s+\gamma$ is an integer in (c), more information is given in \cite{GG14} in terms of H\"older-Zygmund spaces $C^k_*$.
We refer to the expository paper \cite{R1} and to \cite{GG14} for more details.

In case of \emph{parabolic} equations, the interior regularity theory has been developed in the last years, and is now well understood; see \cite{CD14,CK15,FR1,JX15,SS16,Se15}.
However, in such evolutionary setting, the boundary regularity is not so well understood.
The results of \cite{RS1} were extended to the parabolic setting by Fern\'andez-Real and the first author in \cite{FR1}, and more recently the case $f\in L^p$ was treated by Grubb in \cite{GG17}.
These results correspond to the case (a) above, and yield at best an expansion near the boundary of order $2s-\varepsilon$ in space and $1-\varepsilon$ in time.
No higher order boundary regularity result like those in (b) or (c) above is known in the parabolic setting.

It is important to notice that parabolic estimates for nonlocal equations offer new specific challenges over elliptic ones. 
Indeed, this already happens in the setting of interior regularity, where even solutions to the fractional heat equation $\partial_tu+(-\Delta)^su=0$ in $B_1$ are always $C^\infty$ in space but in general not $C^1$ in time.

In the context of boundary regularity, a related counterexample was recently found by G. Grubb \cite{GG18}: solutions to $\partial_tu+(-\Delta)^su=f$ in $(0,T)\times \Omega$, with $u=0$ in $(0,T)\times\Omega^c$, do not satisfy in general $u/d^s\in C^{s+\varepsilon}_x(\overline\Omega)$ for any $\varepsilon>0$, not even when $f$ and $\Omega$ are $C^\infty$. 
This is in sharp contrast with the elliptic case, in which solutions satisfy $u/d^s\in C^\infty(\overline\Omega)$.

The aim of this paper is to extend the boundary regularity result (b) above to the parabolic setting.
For this, we develop a parabolic version of the boundary blow-up methods developed by the first author and Serra in \cite{RS2}.
Our main result gives an expansion near the boundary of order $2s+\gamma$ in space and $1+\gamma/2s$ in time, for $\gamma\in(0,s)$, and reads as follows.

\begin{thm} \label{thm.main}
Let $s\in(0,1)$, and $\gamma\in(0,s)$ such that $\gamma+s$ is not an integer.
Let $\Omega$ be any $C^{2,\gamma}$ domain, and $u(t,x)$ be any solution of
\begin{equation}\label{eq.localest}
  \left\{ \begin{array}{rcll}
  \de_t u - L u & = & f & \textrm{in }(-1,0)\times(B_1\cap \Omega) \\
   u & = & 0 & \textrm{in } (-1,0)\times(B_1\setminus \Omega)
  \end{array}\right.
\end{equation}
with $L$ of the form \eqref{eq.L}.
Assume $a\in C^{1,\gamma}(S^{n-1})$ and $f\in C_{t,x}^{\frac{\gamma}{2s},\gamma}((-1,0)\times\overline{\Omega})$.
Let
\[C_0=\|u\|_{C_{t}^{\gamma/2s}((-1,0)\times\R^n)}+\|f\|_{C_{t,x}^{\gamma/2s,\gamma}((-1,0)\times\overline{\Omega})}.\]
Then,
\begin{equation}\label{eq.estimate}
\|\partial_tu\|_{C_{t,x}^{\gamma/2s,\gamma}\left(\left(-\frac{1}{2},0\right)\times\overline{\Omega}\cap B_{1/2}\right)} + \left\| u/d^s \right\|_{C_{t,x}^{1/2+\gamma/2s,s+\gamma}\left(\left(-\frac{1}{2},0\right)\times\overline{\Omega}\cap B_{1/2}\right)} \leq CC_0.
\end{equation}
The constant $C$ depends only on $\gamma$, $\|a\|_{C^{1,\gamma}(S^{n-1})}$, $n$, $s$, and the ellipticity constants.
\end{thm}

Notice this is an estimate of order $2s+\gamma$ in space and $1+\gamma/2s$ in time.
This is precisely the order one expects for this equation.
In fact, for the local case (say when $L=\Delta$) one actually has $u\in C_x^{2+\gamma}$ and $u\in C_t^{1+\gamma/2}$.
Our estimate \eqref{eq.estimate} extends for the first time the results of \cite{RS2} to the parabolic setting  (in case of linear equations).
It is new even for the fractional Laplacian in $C^\infty$ domains.
Furthermore, in view of the above mentioned counterexample of Grubb, our Theorem is sharp (or almost-sharp) in the sense that it cannot hold for $\gamma>s$.

Notice also that the assumption that $u\in C_{t}^{\gamma/2s}((-1,0)\times\R^n)$ is necessary, even for the interior regularity.
Indeed, in \cite{FR1} the first author and Fern\'andez-Real showed that for any $\varepsilon>0$
one can construct a solution to the homogeneous fractional heat equation $\partial_tu+(-\Delta)^su=0$ in $(-1,0)\times B_1$ which is in $C_{t}^{\gamma/2s-\varepsilon}((-1,0)\times\R^n)$ and not
$C_t^{1+\gamma/2s}$ in $(-1,0)\times B_{1/2}$.

As for the regularity assumption on the function $a$ in \eqref{eq.L} and the domain $\Omega$, it is needed in the proof of Proposition \ref{prop.mainprop}. 
These are the same assumptions as in the elliptic case.
Indeed, since $L(d^s)$ does not vanish identically (something that does happen in the case of a flat boundary, where we have $L(x_n)_+^s=0$) we need to control over $[L(d^s)]_{C^\gamma}$. Such an estimate is obtained in terms of the $C^{1,\gamma}$ norm of $a$ and the $C^{2,\gamma}$ norm of $\Omega$; see Lemma 8.3 in \cite{RS2}.

As an immediate consequence of Theorem \ref{thm.main}, we get the following result for the Dirichlet problem.

\begin{cor}\label{cor.mainboundary}
Let $s\in(0,1)$, and $\gamma\in(0,s)$ such that $\gamma+s$ is not an integer.
Let $L$ be any operator of the form \eqref{eq.L} and $\Omega$ be any bounded $C^{2,\gamma}$ domain.
Suppose that $u(t,x)$ is the solution of
\begin{equation}\label{eq.fracheat_u0}
  \left\{ \begin{array}{rcll}
  \de_t u - L u&=&f& \textrm{in }\Omega,\ T > t > 0 \\
  u&=&0& \textrm{in } \R^n\setminus \Omega,\ T > t \geq 0, \\
  u(0,\cdot)&=&u_0& \textrm{in } \Omega,\ t = 0, \\
  \end{array}\right.
\end{equation}
Assume in addition that $a\in C^{1,\gamma}(S^{n-1})$, and denote
\[
C_0= \|u_0\|_{L^2(\Omega)} + \|f\|_{C_{t,x}^{\gamma/2s,\gamma}((0,T)\times\Omega)}.
\]
Then, we have
\begin{equation}
\|\partial_tu\|_{C^{\gamma/2s,\gamma}_{t,x}\left(\left(t_0,T\right)\times \overline{\Omega}\right)} + \left\| u/d^s \right\|_{C^{1/2+\gamma/2s,\gamma+s}_{t,x}\left(\left(t_0,T\right)\times\overline{\Omega}\right)} \leq CC_0,
\end{equation}
for any $0<t_0< T$ where $C$ depends only on $n$, $s$, $\Omega$, $t_0$, $T$, $\|a\|_{C^{1,\gamma}(S^{n-1})}$ and the ellipticity constants.
\end{cor}

\begin{rem}\label{rmk.seminorm}
Notice that $\gamma+s$ could be bigger than 1 (and is always less than 2).
This calls for some care when defining the H\"older seminorm being used, as well as the full norm.
When $\gamma+s<1$ there is no issue and the norm is defined as usual, i.e.
\[
 [w]_{C_{t,x}^{\frac{s+\gamma}{2s},s+\gamma}(I\times\Omega)}=\sup_{\substack{t,t'\in I \\ x,x'\in\Omega}}\frac{| w(t,x)- w(t',x')|}{|t-t'|^{\frac{s+\gamma}{2s}} + |x-x'|^{s+\gamma}}
 \]
and
\[
  \|w\|_{{C_{t,x}^{\frac{s+\gamma}{2s},s+\gamma}}(I\times\Omega)}=\|w\|_{L^\infty(I\times\Omega)}+[w]_{C_{t,x}^{\frac{s+\gamma}{2s},s+\gamma}(I\times\Omega)}.
\]

However, when $\gamma+s>1$ we need a higher order H\"older seminorm.

As in \cite{FR1}, we define in this case
\[
 [w]_{C_{t,x}^{\frac{s+\gamma}{2s},s+\gamma}(I\times\Omega)}=[w]_{C_t^{\frac{s+\gamma}{2s}}(I\times\Omega)}+[\nabla w]_{C_{t,x}^{\frac{\gamma+s-1}{2s},\gamma+s-1}(I\times\Omega)}
\]
where
\[
  [w]_{C_t^\alpha(I\times\Omega)}=\sup_{x\in\Omega}[w(\cdot,x)]_{C^\alpha(I)}
\]
and
\[
  \|w\|_{{C_{t,x}^{\frac{s+\gamma}{2s},s+\gamma}}(I\times\Omega)}=\|w\|_{L^\infty(I\times\Omega)}+\|\nabla w\|_{L^\infty(I\times\Omega)}+[w]_{C_{t,x}^{\frac{s+\gamma}{2s},s+\gamma}(I\times\Omega)}.
\]

Finally, in Proposition \ref{prop.interior} we are going to need the case $\alpha+\beta+s$ possibly bigger than 2. So in that case, again as in \cite{FR1}, we set

\[
[w]_{C^{\frac{\alpha+\beta+s}{2s},\alpha+\beta+s}(I\times \Omega)} = [\de_t w]_{C_{t,x}^{\frac{\alpha+\beta-s}{2s},\alpha+\beta-s}(I\times \Omega)} +[\nabla w]_{C_t^{\frac{\alpha+\beta-s}{2s},}(I\times \Omega)} +  [D^2_x w]_{C^{\frac{\alpha+\beta-3s}{2s},\alpha+\beta+s-2}_{t,x}(I\times \Omega)}
\]
and
\[
  \|w\|_{{C_{t,x}^{\frac{s+\gamma}{2s},s+\gamma}}(I\times\Omega)}=\|w\|_{L^\infty(I\times\Omega)}+\|\nabla w\|_{L^\infty(I\times\Omega)}+\|D^2 w\|_{L^\infty(I\times\Omega)}+[w]_{C_{t,x}^{\frac{s+\gamma}{2s},s+\gamma}(I\times\Omega)}.
\]

\end{rem}

To prove Theorem \ref{thm.main} we extend the ideas from \cite{RS2} to the parabolic setting.
Namely, we need to develop a parabolic version of the higher order boundary blow-up (elliptic) methods of \cite{RS2},
and then combine this with appropriate interior estimates in order to get the fine regularity estimate for $u$ up to the boundary.

\begin{rem}
Besides its own interest, sharp boundary regularity estimates find usually applications in free boundary problems; see for example \cite{DS16}.
We think that the ideas introduced in this paper could be used in order to establish the higher regularity of free boundaries in the obstacle problem for the fractional Laplacian,
at least in case $s>1/2$; see \cite{BFR2}.
\end{rem}

The paper is organized as follows.
We first prove a Liouville-type theorem in Section \ref{sec2}.
Then, in Section \ref{sec3} we prove the main proposition, namely Proposition \ref{prop.mainprop}.
Finally, in Section \ref{sec4} we prove Theorem~\ref{thm.main}, and in the Appendix we prove Proposition \ref{prop.interior}.

\section{A Liouville-type theorem in the half space}\label{sec2}

In this section we prove a Liouville type theorem in the half space for nonlocal parabolic equations. Let us define first the parabolic cylinders:
\[
 Q_r(t,x)=(t-r^{2s},t)\times B_r(x),
\]
where $B_r(x)$ is the ball of radius $r$ around $x$. Moreover, we will denote $Q_r(0,0)$ by $Q_r$ and $(-r^{2s},0)\times (B_r\cap\{x_n>0\})$ by $Q_r^+$.

\begin{thm}\label{thm.liou}
Let $s\in(0,1)$, $\gamma\in(0,s)$, $\beta\in(\gamma,1)$ and $\alpha=s+\gamma-\beta\in (0,s)$.
Let $w$ satisfy
\[
\left\{\begin{array}{rcl}
    (\partial_t-L)(w(\cdot+\tau,\cdot+h)-w(\cdot,\cdot)) & = & 0 \:\textrm{ in }\: (-\infty,0)\times \mathbb{R}^n_+ \\
                  w & = & 0 \:\textrm{ in }\: (-\infty,0)\times \mathbb{R}^n_-.
       \end{array}
\right.
\]
for $h\in \R^n,h_n\geq0$, $\tau<0$ where $L$ is an operator of the form \eqref{eq.L} and assume that $w$ satisfies the growth condition
\[
[w/(x_n)_+^s]_{C_{t,x}^{\beta/2s,\beta}(Q_R^+)}\leq CR^\alpha,\quad R\geq 1
\]
for some positive constant $C$. Then
\[
w(x)=(x_n)_+^s(p\cdot x+q)
\]
for some $p\in\R^n$ and $q\in\R$. If $\alpha+\beta<1$ then $p=0$.
\end{thm}

\begin{proof}
First let $h_n=0$ and $v(t,x)=w(t+\tau,x+h)-w(t,x)$. Then $v$ satisfies
\[
\left\{\begin{array}{rcl}
    \partial_tv-Lv & = & 0 \:\textrm{ in }\: (-\infty,0)\times \mathbb{R}^n_+ \\
             v & = & 0 \:\textrm{ in }\: (-\infty,0)\times \mathbb{R}^n_-.
       \end{array}
\right.
\]

Moreover,
\[
\|v(t,x)/(x_n)_+^s\|_{L^\infty(Q_R)}\leq CR^\alpha,
\]
so
\[
\|v\|_{L^\infty(Q_R)}\leq CR^{\alpha+s}
\]
for $ R\geq 1$. Then we can apply Theorem 4.11 in \cite{FR1} (note that $\alpha+s<2s$) to get that
\[
v(t,x)=K(x_n)_+^s
\]
for some constant $K$. In particular, this implies that $w$ is constant as a function of time and
\[
w(t,x+h)-w(t,x)=K(x_n)_+^s
\]
for any $h$ such that $h_n=0$. Therefore,
\[
w(t,x)=K(x_n)_+^s(a\cdot x+b)+\psi(x_n)
\]
for some one dimensional function $\psi$.

Now, by Lemma 5.5 in \cite{RS2} and the fact that $w$ is constant in time and solves the equation we have
\[
L(w(t,\cdot+h)-w(t,\cdot))=L(\psi(\cdot +h_n)-\psi(\cdot))=0\qquad \textrm{in}\quad \mathbb{R}_+^n
\]
for $h_n\geq 0$ so $\psi$ satisfies:
\[
\left\{\begin{array}{rcl}
		 L(\psi(\cdot +h_n)-\psi(\cdot))& = & 0 \:\textrm{ in }\: \mathbb{R}_+ \\
                  \psi & = & 0 \:\textrm{ in }\: \mathbb{R}_-.
       \end{array}
\right.
\]

Moreover, because of the hypothesis on $w$, $\psi$ has the growth control
\[
[\psi/(x_n)_+^s]_{C^\beta([0,R])}\leq CR^\alpha,\quad R\geq 1.
\]
so we can apply Lemma 5.3 in \cite{RS2} to get
\[
\psi(x_n)=C_1(x_n)_+^s+C_2(x_n)_+^{1+s},
\]
and the result follows.
\end{proof}

\section{Main proposition}\label{sec3}

In this section we take the main step towards the proof of Theorem \ref{thm.main}, namely we prove Proposition \ref{prop.mainprop}. We start with th following technical lemma that allows
us to take limits.

\begin{lem}\label{lem.limwm}
Let $s\in(0,1)$ and $\{L_m\}_{m\in\N}$ be a sequence of operators of the form  \eqref{eq.L}.
Let $\{w_m\}_{m\in\N}$ and $\{f_m\}_{m\in\N}$ be sequences of functions satisfying
\[
(\de_t-L)(w_m(\cdot+h,\cdot+\tau)-w_m(\cdot,\cdot)) = f_m \textrm{ in } I\times K
\]
in the weak sense for a given bounded interval $I\subset (-\infty,0]$ and a bounded domain $K \subset\R^n$.
Assume that $L_m$ are given by \eqref{eq.L} for some $a_m:S^{n-1}\mapsto[\lambda,\Lambda]$ that
converge uniformly to some $a$. Let $L$ be the operator associated to such $a$, and suppose that, for some functions $w$ and $f$, the following hypotheses hold:
\begin{enumerate}
\item $w_m$ converges to $w$ uniformly in compact sets of $(-\infty,0]\times\R^n$,
\item $f_m \to f$ uniformly on $I\times K$,
\item $\sup_{t\in I}|w_m(t+\tau,x+h)-w_m(t,x)| \leq C\left(1+|x|^{2s-\varepsilon}\right)$ for some $\varepsilon > 0$, and for all $x\in \R^n$.
\end{enumerate}
Then, $w$ satisfies
\[
(\de_t-L)(w(\cdot+h,\cdot+\tau)-w(\cdot,\cdot)) = f \textrm{ in } I\times K
\]
in the weak sense.
\end{lem}

\begin{proof}
Using the definition of weak solution we get, for each $m$ and $\eta\in C^\infty_c(I\times K)$:
\[
\int_{I\times K}(w_m(\cdot+h,\cdot+\tau)-w_m(\cdot,\cdot))(-\de_t-L_m)\eta\:dtdx = \int_{I\times K}f_m\eta\:dtdx.
\]
Because of the uniform convergence $a_m$ to $a$ we have that $(-\de_t-L_m)\eta$ converges to $(-\de_t-L)\eta$ thanks to the dominated convergence theorem.
Also, since $\eta$ is compactly supported we easily find $| L_m\eta(x)|\leq C(1+| x|^{n+2s})^{-1}$, so we just need to prove a growth condition on the difference
quotients that allows us to pass to the limit, i.e. we want
\[
\bigl|(w_m(\cdot+h,\cdot+\tau)-w_m(\cdot,\cdot))(-\de_t-L_m)\eta\bigr|\leq \frac{C}{1+|x|^{n+\varepsilon}}
\]
for some positive $\varepsilon$.
Because of the previous observation, this ammounts to show that
\[
\bigl|(w_m(\cdot+h,\cdot+\tau)-w_m(\cdot,\cdot))\bigr|\leq C(1+|x|^{2s-\varepsilon})
\]
but since that is exactly the third assumption, the result follows.
\end{proof}

Next we are going to prove Proposition \ref{prop.mainprop}.
First, let us give the following definition:

\begin{defi}
 We say that $\Gamma$ is a $C^{2,\gamma}$ surface splitting $\R^n$ into $\Omega^+$ and $\Omega^-$ with norm smaller than 1 if the following holds:
 \begin{itemize}
  \item $\Gamma\subset\R^n$ is the graph of a $C^{2,\gamma}$ function with $C^{2,\gamma}$ norm less than 1
  \item the two disjoint domains $\Omega^+$ and $\Omega^-$ partition $\R^n$, i.e. $\R^n=\overline{\Omega^+}\cup\overline{\Omega^-}$ and $\Omega^+\cap\Omega^-=\emptyset$
  \item $\Gamma=\partial\Omega^+=\partial\Omega^-$ and $0\in\Gamma$
  \item $\nu(0)=e_n$ where $\nu$ is the normal vector to $\Gamma$.
 \end{itemize}
Moreover, we will denote by $d(x)$ any $C^{2,\gamma}(\overline{\Omega^+})$ function that coincides with dist$(x,\Omega^-)$ in a neighborhood of $\Gamma\cap B_4$ and vanishes outside $B_5$.
\end{defi}

Then we have the following result:

\begin{prop}
\label{prop.mainprop}
Let $s\in(0,1)$, $\gamma\in(0,s)$, $\beta\in(\gamma,1)$ and $\alpha=s+\gamma-\beta\in(0,s)$.
Let $\Gamma$ be any $C^{2,\gamma}$ surface splitting $\R^n$ into $\Omega^+$ and $\Omega^-$ with norm smaller than 1.
Let $u$ be any solution of
\begin{equation}\label{eq.eqmainprop}
  \left\{ \begin{array}{rcll}
  \de_t u - L u & = & f & \textrm{ in }(-1,0)\times (B_1\cap\Omega^+) \\
   u & = & 0 & \textrm{in } (-1,0)\times \Omega^-
  \end{array}\right.
\end{equation}
with $L$ an operator of the form \eqref{eq.L}. 
Assume that
\begin{equation}\label{eq.betanorm}
[u/d^s]_{C_{t,x}^{\beta/2s,\beta}((-1,0)\times\Omega^+)}\leq 1 \quad\textrm{ and }\quad \|f\|_{C_{t,x}^{\gamma/2s,\gamma}((-1,0)\times (B_1\cap\Omega^+))}\leq 1
\end{equation}
and $a\in C^{1,\gamma}(S^{n-1})$.

Then for any $r>0$
\begin{equation}\label{eq.pr}
r^{-\alpha}\bigl[u/d^s-(p_r\cdot x+q_r)\bigr]_{C_{t,x}^{\beta/2s,\beta}((-r^{2s},0)\times (B_r\cap\Omega^+))}\leq C
\end{equation}
where $C$ is a constant depending only on $n$, $s$, $\alpha$, $\beta$, $\|a\|_{C^{1,\gamma}(S^{n-1})}$, and the ellipticity constants and $p_r\cdot x+q_r$ is the least square approximation
of $u/d^s$ in $(-r^{2s},0)\times (B_r\cap\Omega^+)$, i.e.
\begin{equation}\label{eq.min}
p_r\cdot x+q_r=\arg\min_{\mathcal{P}} \int_{-r^{2s}}^0\int_{B_r\cap\Omega^+}\bigl(u/d^s-(p\cdot x+q)\bigr)^2\:dx\:dt
\end{equation}
where $\mathcal{P}$ is the space of all polynomials of degree at most $\lfloor \alpha+\beta\rfloor$ (i.e., $p_r=0$ if $\alpha+\beta<1$).
\end{prop}

\begin{proof}
The proof is by contradiction. Assume \eqref{eq.pr} does not hold. Then there are sequences $\Gamma_k$, $u_k$, $L_k$ and $f_k$ satisfying:
\begin{itemize}
 \item $\Gamma_k$ is a $C^{2,\gamma}$ surface splitting $\R^n$ into $\Omega_k^+$ and $\Omega_k^-$ with norm less than 1
 \item $L_k$ is of the form \eqref{eq.L} and $a_k\in C^{1,\gamma}(S^{n-1})$
 \item $[u_k/d_k^s]_{C_{t,x}^{\beta/2s,\beta}((-1,0)\times\Omega_k^+)}\leq 1$ and $\|f_k\|_{C_{t,x}^{\gamma/2s,\gamma}((-1,0)\times (B_1\cap\Omega_k^+))}\leq 1$
 \item $u_k$ is a solution of
 \begin{equation}\label{eq.eqk}
  \left\{ \begin{array}{rcll}
  \de_t u_k - L_k u_k & = & f_k & \textrm{in }(-1,0)\times (B_1\cap\Omega_k^+) \\
   u_k & = & 0 & \textrm{in } (-1,0)\times \Omega_k^-.
  \end{array}\right.
\end{equation}
\end{itemize}
such that
\[
\sup_k\:\sup_{r>0}\:r^{-\alpha}\bigl[u_k/d_k^s-(p_{r,k}\cdot x+q_{r,k})\bigr]_{C_{t,x}^{\beta/2s,\beta}((-r^{2s},0)\times (B_r\cap\Omega_k^+))}=\infty
\]
where
\begin{equation}\label{eq.mincond}
p_{r,k}\cdot x+q_{r,k}=\arg\min_{\mathcal{P}} \int_{-r^{2s}}^0\int_{B_r\cap\Omega_k^+}(u_k/d_k^s-(p\cdot x+q))^2\:dx\:dt.
\end{equation}

Define
\[
\theta(r):=\sup_k\:\sup_{r'>r}\:(r')^{-\alpha}\bigl[u_k/d_k^s-(p_{r',k}\cdot x+q_{r',k})\bigr]_{C_{t,x}^{\beta/2s,\beta}((-(r')^{2s},0)\times (B_{r'}\cap\Omega_k^+))}.
\]
Notice that $\theta$ is a nondecreasing function of $r$ that goes to $\infty$ as $r\to0$.
Moreover, since for any fixed $r$ we have $\theta(r)<\infty$ we can find a sequence $r'_m\geq 1/m$, $r'_m\searrow 0$ and $k_m$ such that
\[
(r_m')^{-\alpha}\bigl[u_{k_m}/d_{k_m}^s-p_{r_m',k_m}\cdot x-q_{r_m',k_m}\bigr]_{C_{t,x}^{\beta/2s,\beta}((-(r'_m)^{2s},0)\times (B_{r'_m}\cap\Omega_{k_m}^+))}\geq \frac{1}{2}\theta(1/m)\geq\frac{1}{2}\theta(r'_m).
\]

Let us denote
\[
u_m=u_{k_m}, \quad p_m=p_{r_m',k_m}, \quad\textrm{ and }\quad q_m=q_{r_m',k_m}
\]
and
\[
\overline{\Gamma}_m=\frac{1}{r'_m}\Gamma_{k_m}, \quad \overline{\Omega}^+_m=\frac{1}{r'_m}\Omega^+_{k_m}, \quad\textrm{ and }\quad \bar{d}_m(x)=\frac{1}{r'_m}d_{k_m}(r'_mx).
\]
Notice that $\overline{\Gamma}_m$ converges to $\{x_n=0\}$, $\overline{\Omega}^+_m$ converges to $\R^n_+$ and $\bar{d}_m^s$ converges locally uniformly to $(x_n)_+^s$ as $m\rightarrow\infty$. Further, let us simplify the notation by defining
\[
Q_{r,k}^+=(-r^{2s},0)\times (B_r\cap\Omega_k^+)\quad\textrm{ and }\quad\overline{Q}_{r,m}^+=(-r^{2s},0)\times (B_r\cap\overline{\Omega}_m^+).
\]

We define the blow-up sequence:
\[
v_m(t,x)=\frac{u_m((r'_m)^{2s}t,r'_mx)/(r'_m\bar{d}_m(x))^s-(p_m\cdot (r'_mx)+q_m)}{(r'_m)^{\alpha+\beta}\theta(r'_m)}.
\]
Notice that
\begin{equation}\label{eq.boundvbel}
[v_m]_{C_{t,x}^{\beta/2s,\beta}(\overline{Q}_{1,m}^+)}\geq 1/2
\end{equation}
and that
\begin{equation}\label{eq.ort}
\int_{-1}^0\int_{B_1\cap\overline{\Omega}_m^+}v_m(p\cdot x+q)\:dx\:dt=0
\end{equation}
for all $p\in\mathbb{R}^n$ and $q\in\mathbb{R}$ because of the minimization condition \eqref{eq.mincond}.

Let us show next that
\begin{equation}\label{eq.boundv}
[v_m]_{C^{\beta/2s,\beta}(\overline{Q}_{R,m}^+)}\leq CR^\alpha
\end{equation}
for $1\leq R\leq \frac{1}{r_m'}$.
For this need an estimate of the form
\begin{equation}\label{eq.linear}
\bigl[p_{k_m,Rr'_m}\cdot x-q_{k_m,Rr'_m}-(p_{k_m,r'_m}\cdot x+q_{k_m,r'_m})\bigr]_{C_{t,x}^{\beta/2s,\beta}(Q_{r'_mR,k_m}^+)}\leq (Rr'_m)^\alpha\theta(r'_m)
\end{equation}
which actually amounts to bound $\vert p_{k,Rr}-p_{k,r}\vert$.
This will be achieved via a dyadic estimate, i.e. for $R=2^k$, $k\geq1$ (notice that the estimate holds if $R\in[1,1/r'_m)$ by the definition of $v_m$).
We will also prove here only the case where $\alpha+\beta>1$, the case $\alpha+\beta<1$ is analogous.
Let us first estimate the difference when $R=2$:
\begin{eqnarray*}
\frac{\vert p_{k,2r}-p_{k,r}\vert r^{1-\beta}}{r^\alpha\theta(r)} & \leq & \frac{[(p_{k,2r}-p_{k,r})\cdot x]_{C_{t,x}^{\beta/2s,\beta}(Q_{r,k}^+)}}{r^\alpha\theta(r)} \\
                                                                  & = & \frac{[p_{k,2r}\cdot x+q_{k,2r}-p_{k,r}\cdot x-q_{k,r}]_{C_{t,x}^{\beta/2s,\beta}(Q_{r,k}^+)}}{r^\alpha\theta(r)} \\
                                                                  & \leq & \frac{2^\alpha\theta(2r)}{\theta(r)}\frac{[u_k/d_k^s-p_{k,2r}\cdot x-q_{k,2r}]_{C_{t,x}^{\beta/2s,\beta}(Q_{2r,k}^+)}}{(2r)^\alpha\theta(2r)} \\
								 &  & +\frac{[u_k/d_k^s-p_{k,r}\cdot x-q_{k,r}]_{C_{t,x}^{\beta/2s,\beta}(Q_{r,k}^+)}}{r^\alpha\theta(r)} \\
                                                                  & \leq & C,
\end{eqnarray*}
where we have used the definition of $\theta$ and its monotonicity.
Now if $R=2^k$ we have
\begin{eqnarray*}
\vert p_{k,2^kr}-p_{k,r}\vert & \leq & \sum_{j=1}^k \vert p_{k,2^jr}-p_{k,2^{j-1}r}\vert \\
				   & \leq & C\sum_{j=1}^k (2^{j-1}r)^{\alpha+\beta-1}\theta(2^{j-1}r) \\
				   & \leq & C\theta(r)(2^kr)^{\alpha+\beta-1},
\end{eqnarray*}
using that $\alpha+\beta-1>0$ and this readily implies \eqref{eq.linear}. Now we can apply \eqref{eq.linear} as follows:
\begin{eqnarray*}
[v_m]_{C_{t,x}^{\beta/2s,\beta}(\overline{Q}_{R,m}^+)} & = & \frac{1}{(r'_m)^{\alpha}\theta(r'_m)}[u_m/d_m^s-(p_m\cdot x+q_m)]_{C_{t,x}^{\beta/2s,\beta}(Q_{r'_mR,k_m}^+)} \\
			& = & \frac{R^\alpha}{(Rr'_m)^{\alpha}\theta(r'_m)}[u_m/d_m^s-(p_m\cdot x+q_m)]_{C_{t,x}^{\beta/2s,\beta}(Q_{r'_mR,k_m}^+)} \\
                         & \leq & \frac{R^\alpha}{(Rr'_m)^\alpha\theta(r'_m)}[u_m/d_m^s-(p_{k_m,Rr'_m}\cdot x+q_{k_m,Rr'_m})]_{C_{t,x}^{\beta/2s,\beta}(Q_{r'_mR,k_m}^+)}\\
                         &  & +\frac{R^\alpha}{(Rr'_m)^\alpha\theta(r'_m)}[p_{k_m,Rr'_m}\cdot x-q_{k_m,Rr'_m}-(p_m\cdot x+q_m)]_{C_{t,x}^{\beta/2s,\beta}(Q_{r'_mR,k_m}^+)} \\
                         & \leq & \frac{R^\alpha\theta(Rr'_m)}{\theta(r'_m)}+C\frac{R^\alpha}{(Rr'_m)^\alpha\theta(r'_m)}(Rr'_m)^\alpha\theta(r'_m)\\
			& \leq & CR^\alpha,
\end{eqnarray*}
and we get \eqref{eq.boundv}.

When $R=1$, \eqref{eq.boundv} implies that the oscillation of $v_m$ in $Q_1^+$ is bounded by some universal constant $C$, which in turn implies
\[
\|v_m-M\|_{L^\infty{(Q_{1,m}^+)}}\leq C
\]
for some $M$. This gives, by \eqref{eq.ort}
\[
\|v_m\|_{L^\infty{(Q_{1,m}^+)}}\leq C.
\]
Finally, the bound on the H\"older seminorm \eqref{eq.boundv} and the $L^\infty$ bound in $Q_1^+$ give
\begin{equation}\label{eq.unifv}
\|v_m\|_{L^\infty{(Q_{R,m}^+)}}\leq CR^{\alpha+\beta}.
\end{equation}

By the Arzel\`a-Ascoli theorem, the previous bounds imply that a subsequence of $\{v_m\}$ converges uniformly on compact subsets of $(-\infty,0)\times\mathbb{R}^n$ to
a uniformly continuous function $v$. Let
\begin{equation}\label{eq.wm}
w_m(t,x):=v_m(t,x)\bar{d}_m^s(x)=\frac{u_m((r'_m)^{2s}t,r'_mx)}{(r'_m)^{\alpha+\beta+s}\theta(r'_m)}-\frac{\bar{d}_m^s(x)(p_m\cdot r'_mx+q_m)}{(r'_m)^{\alpha+\beta}\theta(r'_m)}
\end{equation}
and $w(t,x)=v(t,x)(x_n)_+^s$.

We claim that $w$ satisfies the hypothesis of the Liouville-type Theorem \ref{thm.liou}. To see this, note that, by \eqref{eq.eqk} (and for fixed $h$ and $\tau$) and the scaling of the equation
\begin{eqnarray*}
\frac{(\partial_t-L_{k_m})(u_m((r'_m)^{2s}(t+\tau),r'_m(x+h))-u_m((r'_m)^{2s}t,r'_mx))}{(r'_m)^{\alpha+\beta+s}\theta(r'_m)}\\
=\frac{f_{k_m}((r'_m)^{2s}(t+\tau),r'_m(x+h))-f_{k_m}((r'_m)^{2s}t,r'_mx)}{(r'_m)^{\alpha+\beta-s}\theta(r'_m)} \leq \frac{C[f_{k_m}]_{C_{t,x}^{\gamma/2s,\gamma}(Q_1^+)}}{(r'_m)^{\alpha+\beta-s-\gamma}\theta(r'_m)} 	
\end{eqnarray*}
in $\overline{\Omega}_m^+$ and this last term goes to 0 as $m\rightarrow \infty$ (recall $\alpha+\beta=s+\gamma$).

On the other hand, using Lemma 8.3 in \cite{RS2} and the definition of $\bar{d}_m$ we get
\begin{eqnarray*}
L_{k_m}(\bar{d}_m^s(x+h)(p_m\cdot r'_m(x+h)+q_m)-\bar{d}_m^s(x)(p_m\cdot r'_mx+q_m)) \\
\leq C(r'_m)^{s+\gamma}(|p_m|+|q_m|).
\end{eqnarray*}
So we need to get the appropriate bounds on $|p_m|$ and $|q_m|$. But with the same proof as in Proposition 8.2 in \cite{RS2} we get that, for $r\in[2^{-i},2^{-i+1}]$
\[
 \frac{|p_{k,r}|+|q_{k,r}|}{\theta(r)}\leq C\sum_{j=0}^i\frac{\theta(2^{-j})}{\theta(r)}(1/2)^{j(\alpha+\beta-1)}.
\]
which goes to 0 as $r\searrow0$. Hence, recalling \eqref{eq.wm}, we have that
\[
(\partial_t-L_{k_m})(w_m(t+\tau,x+h)-w_m(t,x))\longrightarrow 0
\]
uniformly on compact sets as $m\rightarrow\infty$. But thanks to Lemma \ref{lem.limwm} (condition 3 follows easily from \eqref{eq.boundv}, \eqref{eq.unifv} and the definition of $\alpha$ and $\beta$),
this implies
\[
\partial_t(w(t+\tau,x+h)-w(t,x))-L(w(t+\tau,x+h)-w(t,x))=0
\]
in $\mathbb{R}^n_+$ for $h\in\mathbb{R}^n,h_n\geq0$, $\tau<0$ and some $L$ of the form \eqref{eq.L}. Also $w=0$ in $\mathbb{R}^n_-$ by the uniform convergence.

Moreover, because of \eqref{eq.boundv} $w$ also satisfies the growth condition
\[
[w/(x_n)_+^s]_{C_{t,x}^{\beta/2s,\beta}(Q_R^+)}\leq CR^\alpha,\quad R\geq 1
\]
so we can apply by Theorem \ref{thm.liou} and we get that $w(t,x)=(x_n)_+^s(p\cdot x+q)$ for some $p\in\R^n$ and $q\in\R$ or, equivalently, $v(t,x)=p\cdot x+q$. This, together
with \eqref{eq.ort} gives that $v\equiv0$ so passing to the limit in \eqref{eq.boundvbel} we get a contradiction.
\end{proof}

The following proposition gives a plane (the same) for all $r$.
\begin{prop}
\label{prop.plane}
Let $\alpha,\beta\in(0,1)$ with $\alpha+\beta\neq1$ and let $v$ satisfy
\begin{equation}
\sup_{r>0}r^{-\alpha}[v-(p_r\cdot x+q_r)]_{C_{t,x}^{\beta/2s,\beta}((-r^{2s},0)\times (B_r\cap\Omega^+))}\leq C_0
\end{equation}
with $p_r=0$ if $\alpha+\beta<1$, if not assume that $|p_1|\leq C_0$. Then there exist $p\in\R^n$ and $q\in\R$ such that
\[
 p=\lim_{r\rightarrow0}p_r\quad\textrm{ and } \quad  q=\lim_{r\rightarrow0}q_r
\]
and for all $r>0$
\[
 \|v-(p\cdot x+q)\|_{L^\infty((-r^{2s},0)\times (B_r\cap\Omega^+)))}\leq CC_0r^{\alpha+\beta} \quad |p|\leq CC_0
\]
where $C$ depends only on $\alpha$ and $\beta$.
\end{prop}

\begin{proof}
The proof is the same as in Lemma 7.4 in \cite{RS2}.
\end{proof}

\section{Proof of the main theorem}\label{sec4}

In this section we prove Theorem \ref{thm.main}. We will need the following decay on the H\"older norm (recall the definition of the seminorm in Remark \ref{rmk.seminorm}):

\begin{prop}\label{prop.interior}
Let $s\in(0,1)$, $\gamma\in(0,s)$, $\beta\in(\gamma,1)$ and $\alpha=s+\gamma-\beta\in(0,s)$.
Let $w$ be a solution of
\[
\partial_tw-Lw=f\quad\textrm{ in }\quad (-1,0)\times B_1(e_n)
\]
with $e_n=(0,\ldots,0,1)$.  Assume that $a\in C^{1,\gamma}(S^{n-1})$, $f\in C^{\frac{\gamma}{2s},\gamma}_{t,x}((-1,0)\times B_1(e_n))$ and
\[
\| w\|_{L^\infty((-\infty,0)\times\R^n)}<\infty, \quad \| w\|_{L^\infty(Q_r)}\leq Cr^{\alpha+\beta+s},\quad [w]_{C_{t,x}^{\beta/2s,\beta}(Q_r(-(2r)^{2s},2re_n))}\leq Cr^{\alpha+s}
\]
for all $r>0$. Then, we have
\[
[w]_{C^{\frac{\alpha+\beta}{2s},\alpha+\beta}_{t,x}(Q_{r/2}(-(2r)^{2s},2re_n))}\leq Cr^s
\]
and
\[
[w]_{C^{\frac{\alpha+\beta+s}{2s},\alpha+\beta+s}_{t,x}(Q_{r/2}(-(2r)^{2s},2re_n))}\leq C.
\]
\end{prop}

We will prove Proposition \ref{prop.interior} in the Appendix. 
The next proposition, combined with the interior estimates, will give the proof of Theorem \ref{thm.main}:

\begin{prop}\label{prop.estimate}
Let $s\in(0,1)$, $\Gamma$ be $C^{2,\gamma}$ surface splitting $\R^n$ into $\Omega^+$ and $\Omega^-$ and $u$ be a solution of
\begin{equation}
  \left\{ \begin{array}{rcll}
  \de_t u - L u & = & f & \textrm{in }(-1,0)\times \Omega^+\cap B_1 \\
   u & = & 0 & \textrm{in } (-1,0)\times\Omega^-
  \end{array}\right.
\end{equation}
with $L$ of the form \eqref{eq.L}.
Assume $a\in C^{1,\gamma}(S^{n-1})$ and $f\in C_{t,x}^{\frac{\gamma}{2s},\gamma}((-1,0)\times\Omega^+\cap B_1)$ with $\gamma\in(0,s)$ and
$\gamma+2s$ and $\gamma+s$ not integers.
Let
\[
C_0=\|u\|_{C_t^{\gamma/2s}((-1,0)\times\R^n)}+\|f\|_{C_{t,x}^{\gamma/2s,\gamma}((-1,0)\times\Omega^+\cap B_1)}.
\]
Then, $\partial_tu\in C_{t,x}^{\gamma/2s,\gamma}((-1/2,0)\times \overline{B_{1/2}})$ and $u/d^s\in C_{t,x}^{1/2+\gamma/2s,s+\gamma}((-1/2,0)\times\overline{\Omega^+}\cap B_{1/4})$ and
\begin{equation}
\|\partial_tu\|_{C_{t,x}^{\gamma/2s,\gamma}\left(\left(-\frac{1}{2},0\right)\times \overline{B_{1/2}}\right)} + \left\| u/d^s \right\|_{C_{t,x}^{1/2+\gamma/2s,s+\gamma}\left(\left(-\frac{1}{2},0\right)\times\overline{\Omega^+}\cap B_{1/4}\right)} \leq CC_0.
\end{equation}
The constant $C$ depends only on $\gamma$, $\|a\|_{C^{1,\gamma}(S^{n-1})}$, $n$, $s$, and the ellipticity constants.
\end{prop}	

\begin{proof}
Let us assume that
\[
\|u\|_{C_t^{\gamma/2s}((-1,0)\times\R^n)}+\|f\|_{C_{t,x}^{\gamma/2s,\gamma}((-1,0)\times\Omega^+\cap B_1)}\leq 1,
\]
then we want to show that
\begin{equation}\label{eq.spatimereg}
\|\partial_tu\|_{C_{t,x}^{\gamma/2s,\gamma}\left(\left(-\frac{1}{2},0\right)\times B_{1/2}\right)} + \left\| u/d^s \right\|_{C_{t,x}^{1/2+\gamma/2s,s+\gamma}\left(\left(-\frac{1}{2},0\right)\times\overline{\Omega^+}\cap B_{1/4}\right)} \leq C.
\end{equation}

We will bound each term separately and combine interior estimates with estimates close to the boundary.

\vspace{2mm}

\noindent{\bf Step 1: spatial regularity.} Let $(t_0,x_0)\in(-1/2,0)\times \overline{\Omega^+}\cap B_{1/4}$,
\[
 2r=\textrm{dist}(x_0,\Gamma)<1/4,
\]
we will start with the following estimate:
\begin{equation}\label{eq.spareg}
\|u/d^s\|_{C_{t,x}^{1/2+\gamma/2s,\gamma+s}(Q_r(t_0,x_0))}\leq C
\end{equation}
for a universal constant $C$. Thanks to Lemma 4.8 in \cite{FR1}, \eqref{eq.spareg} ammounts to bound $[u/d^s]_{C_{t,x}^{1/2+\gamma/2s,\gamma+s}(Q_r(t_0,x_0))}$. We will prove the
case $\gamma+s>1$ which is more difficult, the case $\gamma+s<1$ follows similarly.

By Propositions \ref{prop.mainprop} and \ref{prop.plane}, we have that for
$z\in \Gamma$ such that
\[
\textrm{dist}(x_0,\Gamma)=\textrm{dist}(x_0,z)
\]
there exist $l(x)$ defined by
\[
l(x)=p(t_0,z)\cdot x+q(t_0,z)
\]
such that
\[
\|u/d^s-l\|_{L^\infty(Q_{2r}(t_0,x_0))}\leq Cr^{\alpha+\beta}.
\]
and
\[
[u/d^s-l]_{C_{t,x}^{\beta/2s,\beta}(Q_{2r}(t_0,x_0))}\leq Cr^\alpha
\]
(with $\beta\in(\gamma,1)$ and $\alpha=s+\gamma-\beta\in(0,s)$).
This implies
\begin{equation}\label{eq.int1}
 \|u-d^sl\|_{L^\infty(Q_{2r}(t_0,x_0))}\leq Cr^{\alpha+\beta+s}.
\end{equation}
and
\begin{equation}\label{eq.int2}
[u-d^sl]_{C_{t,x}^{\beta/2s,\beta}(Q_{2r}(t_0,x_0))}\leq Cr^{s+\alpha}.
\end{equation}

Then, the estimates on Proposition \ref{prop.interior} give
\begin{equation}\label{eq.int3}
[u-d^sl]_{C_t^{\frac{\alpha+\beta}{2s}}(Q_r(t_0,x_0))}+[\nabla (u-d^sl)]_{C_{t,x}^{\frac{\nu}{2s},\nu}(Q_r(t_0,x_0))}\leq Cr^s
\end{equation}
with $\nu=\gamma+s-1$.

Let $(t,x),(t',x')\in Q_r(t_0,x_0)$. We want to show
\begin{equation}\label{eq.timesemi}
\bigl|u(t,x)/d^s(x)-u(t',x)/d^s(x)\bigr|\leq C|t-t'|^{\frac{\alpha+\beta}{2s}}
\end{equation}
with $C$ independent of $x$ and
\begin{equation}\label{eq.gradsemi}
\bigl|\nabla (u(t,x)/d^s(x))-\nabla (u(t',x')/d^s(x'))\bigr|\leq C(|t-t'|^{\nu/2s}+|x-x'|^\nu).
\end{equation}

\eqref{eq.timesemi} follows using \eqref{eq.int3} and noting that the only $t$ dependence is by $u$ and that $d^s$ and $r^s$ are comparable. More precisely:
\begin{eqnarray*}
u(t,x)/d^s(x)-u(t',x)/d^s(x) & \leq & d^{-s}(x)[u-d^sl]_{C_t^{\frac{\alpha+\beta}{2s}}(Q_r(t_0,x_0))}|t-t'|^{\frac{\alpha+\beta}{2s}} \\
			    & \leq &  Cd^{-s}(x)r^s|t-t'|^{\frac{\alpha+\beta}{2s}} \\
			    & \leq & C|t-t'|^{\frac{\alpha+\beta}{2s}}.
\end{eqnarray*}

For \eqref{eq.gradsemi} we have
\[
 \nabla (u(t,x)/d^s(x))=\nabla u(t,x)/d^s(x)+u(t,x)\nabla(d^{-s}(x)).
\]
Let us write
\[
 \nabla (u(t,x)/d^s(x))-\nabla (u(t',x')/d^s(x')) = \mathcal{A}+\mathcal{B}+\mathcal{C},
\]
with
\begin{eqnarray*}
 \mathcal{A} & = & \frac{\nabla (u(t,x)-d^s(x)l(x))-\nabla (u(t',x')-d^s(x')l(x'))}{d^s(x)} \\
 \mathcal{B} & = & (d^{-s}(x)-d^{-s}(x'))\nabla (u(t',x')-d^s(x')l(x')) \\
 \mathcal{C} & = & (u(t,x)-d^s(x)l(x))\nabla d^{-s}(x) \\
   &  & - (u(t',x')-d^s(x')l(x'))\nabla d^{-s}(x').
\end{eqnarray*}

Now, by \eqref{eq.int3} we have
\[
 |\mathcal{A}|\leq Cr^sd^{-s}(x)(|t-t'|^{\frac{\nu}{2s}}+|x-x'|^\nu)\leq C(|t-t'|^{\frac{\nu}{2s}}+|x-x'|^\nu)
\]
since $d^s$ and $r^s$ are comparable.

Now we need to bound $\mathcal{B}$:
\[
 |\mathcal{B}|= |(d^{-s}(x)-d^{-s}(x'))\nabla (u(t',x')-d^s(x')l(x'))|.
\]
On one hand, we have
\[
|d^{-s}(x)-d^{-s}(x')|\leq Cr^{-s}
\]
(see Lemma 5.5 in \cite{RS1}). On the other hand, \eqref{eq.int2} and \eqref{eq.int3} imply
\[
|\nabla (u(t',x')-d^s(x')l(x'))|\leq Cr^{\alpha+\beta+s-1}
\]
hence
\[
|\mathcal{B}|\leq Cr^{\alpha+\beta-1}\leq C(|t-t'|^{\frac{\nu}{2s}}+|x-x'|^\nu)
\]
since $r^\delta\sim|t-t'|^{\frac{\delta}{2s}}+|x-x'|^\delta$.

Finally, by a classical interpolation inequality in H\"older spaces,
\begin{eqnarray*}
|\mathcal{C}| & \leq & \|\nabla d^{-s}\|_{L^\infty(B_r(x_0))}[u-d^sl]_{C_{t,x}^{\beta/2s,\beta}(Q_r(t_0,x_0))}(|t-t'|^{\frac{\beta}{2s}}+|x-x'|^\beta) \\
    &  & + \|u-d^sl\|_{L^\infty(Q_r(t_0,x_0))}[\nabla d^{-s}]_{C^\beta(B_r(x_0))}(|t-t'|^{\frac{\beta}{2s}}+|x-x'|^\beta) \\
    & \leq & Cr^{-s-1}r^{\alpha+s}r^\beta+r^{\alpha+\beta+s}r^{-s-\beta}r^\beta\leq Cr^{\alpha+\beta-1}
\end{eqnarray*}
and we get \eqref{eq.gradsemi} and hence \eqref{eq.spareg}.

To pass from \eqref{eq.spareg} to bound the second term on the left hand side of \eqref{eq.spatimereg} just notice that we can cover $(-1/2,0)\times\overline{\Omega^+}\cap\{x_0\in B_{1/4}:\textrm{dist}(x_0,\Gamma)<1/8\}$ by a universal (dependent on $\Gamma$)
number of balls in which \eqref{eq.spareg} holds and use interior estimates for the points in $(-1/2,0)\times\overline{\Omega^+}\cap\{x_0\in B_{1/4}:\textrm{dist}(x_0,\Gamma)\geq 1/8\}$.

\vspace{2mm}

\noindent{\bf Step 2: time regularity.} For the time regularity, Proposition \ref{prop.interior} gives
\[
[u-d^sl]_{C_{t,x}^{\frac{\beta+\alpha+s}{2s},\beta+\alpha+s}(Q_{2r}(t_0,x_0))}\leq C
\]
for any $r$ sufficiently small. 
Since the function $d^sl$ is independent of time, this implies
\[
[\partial_t u]_{C_{t,x}^{\frac{\alpha+\beta-s}{2s},\beta+\alpha-s}(Q_r(t_0,x_0))}\leq C.
\]
Now recalling that $\alpha+\beta=s+\gamma$, we find
\[
[\partial_tu]_{C_{t,x}^{\gamma/2s,\gamma}(Q_r(t_0,x_0))}\leq C.
\]

Hence,
\begin{equation}\label{eq.timeest}
\|\partial_tu\|_{C_{t,x}^{\gamma/2s,\gamma}(Q_r(t_0,x_0))}\leq C.
\end{equation}
The bound on the first term of the left hand side of \eqref{eq.spatimereg} follows similarly as the second term followed from \eqref{eq.spareg}.
\end{proof}

\begin{proof}[Proof of Theorem \ref{thm.main}]
The Theorem follows combining the interior estimates from \cite{FR1} with the boundary estimates from the previous Proposition.
\end{proof}

Finally, we can also give the

\begin{proof}[Proof of Corollary \ref{cor.mainboundary}]
The proof is the same as that of Corollary 1.6 in \cite{FR1} but we sketch it here for completeness.
We can cover $\Omega$ with a (finite, universal) number of balls in which we can apply Proposition \ref{prop.estimate} to get
\[
\|\partial_tu\|_{C_{t,x}^{\gamma/2s,\gamma}\left(\left(1/2,1\right)\times \overline{\Omega}\right)} + \left\| u/d^s \right\|_{C_{t,x}^{1/2+\gamma/2s,s+\gamma}\left(\left(1/2,1\right)\times\overline{\Omega}\right)} \leq CC_1.
\]
with
\[
C_1=\|u\|_{C_{t}^{\gamma/2s}((1/4,1)\times\R^n)}+\|f\|_{C_{t,x}^{\gamma/2s,\gamma}((1/4,1)\times\Omega)}
\]
so we just need to bound $\|u\|_{C_{t}^{\gamma/2s}((1/4,1)\times\R^n)}$ properly (after that the result follows scaling in time). But notice that by Lemma 6.1 in
\cite{FR1} with $c_0 = \|f\|_{L^\infty\left(\left(0,1\right)\times\Omega\right)}$ and $t \geq t_0 = 1/4$,
\begin{align*}
\|u\|_{L^\infty\left(\left(\frac{1}{4},1\right)\times\Omega\right)} & \leq  C\left(\|u(1/4,\cdot)\|_{L^\infty\left(\Omega\right)}+\|f\|_{L^\infty\left(\left(\frac{1}{4},1\right)\times\Omega\right)}\right)\\
& \leq  C\left(\|u(0,\cdot)\|_{L^2\left(\Omega\right)}+\|f\|_{L^\infty\left(\left(0,1\right)\times\Omega\right)}\right).
\end{align*}
Since we have estimates for $\|u\|_{C_{t}^{\gamma/2s}((1/4,1)\times\R^n)}$ given in terms of $\|u\|_{L^\infty\left(\left(\frac{1}{4},1\right)\times\Omega\right)}$ we are done (recall
$u\equiv0$ outside $\Omega$).
\end{proof}

\section{Appendix}

In this Appendix we prove Proposition \ref{prop.interior}. The proof will follow the same compactness argument as the boundary regularity. 
We start with the following Liouville type result:

\begin{lem}\label{lem.liou}
Let $s\in(0,1)$, $\beta\in(0,1)$ and $\alpha\in(0,s)$ satisfying $\alpha+\beta<2s$. 
Let $w$ satisfy
\[
(\partial_t-L)(w(\cdot+\tau,\cdot+h)-w(\cdot,\cdot)) = 0 \:\textrm{ in }\: (-\infty,0)\times \mathbb{R}^n
\]
for all $h\in\R^n$ and $\tau<0$. 
Assume also that
\begin{equation}\label{eq.betagro}
[w]_{C_{t,x}^{\frac{\beta}{2s},\beta}(Q_R)}\leq CR^{\alpha+s}
\end{equation}
for all $R\geq 1$.
Then
\[
w(t,x)=at+x^TAx+p\cdot x+q
\]
for some $p\in\R^n$ and $a,q\in\R$. 
Moreover, if $\alpha+\beta<s$ then $a=0$, if $\alpha+\beta+s<2$ then $A=0$, and if $\alpha+\beta+s<1$ then $p=0$.
\end{lem}

\begin{proof}
 Let $\rho>0$ and define
\[
v(t,x)=\frac{w(\rho^{2s}(t+\tau),\rho (x+h))-w(\rho^{2s}t,\rho x)}{\rho^{\alpha+s}(|\rho\tau|^{\beta/2s}+|\rho h|^{\beta})}.
\]
Note that
\[
 \partial_tv-Lv=0
\]
in $Q_1$ and, by \eqref{eq.betagro},
\[
 \|v\|_{L^{\infty}(Q_R)}\leq CR^{\alpha+s}
\]
for all $R\geq 1$ and in particular
\[
 \|v\|_{L^{\infty}(Q_1)}\leq C.
\]
Then, by the interior estimates of Theorem 1.3 in \cite{FR1} with can get, for $\theta>\alpha+s$
\[
 \|v\|_{C_{t,x}^{\theta/2s,\theta}(Q_{1/2})}\leq C.
\]
Hence, the incremental quotients of order $\beta$ of $w$ are uniformly bounded in $C_{t,x}^{\theta/2s,\theta}$. 
This implies (see \cite{CC}, Lemma 5.6) $w\in C_{t,x}^{\frac{\theta+\beta}{2s},\theta+\beta}$ and scaling back we get
\[
 [w]_{C_{t,x}^{\frac{\alpha+\beta+s}{2s},\alpha+\beta+s}(Q_\rho)}\leq [w]_{C_{t,x}^{\frac{\theta+\beta}{2s},\theta+\beta}(Q_\rho)}\leq \rho^{\alpha+s-\theta}
\]
and when we let $\rho\rightarrow\infty$ we get
\[
 [w]_{C_{t,x}^{\frac{\alpha+\beta+s}{2s},\alpha+\beta+s}((-\infty,0)\times\R^n)}=0.
\]
This gives the desired result.
\end{proof}

The next proposition is the main tool to prove Proposition \ref{prop.interior}.

\begin{prop}\label{prop.blowup}
Let $s\in(0,1)$, $\gamma\in(0,s)$, $\beta\in(\gamma,1)$ and $\alpha=s+\gamma-\beta\in(0,s)$.
Let $w$ be a solution of
\[
\partial_tw-Lw=f\quad\textrm{ in } Q_1
\]
with $L$ an operator of the form \eqref{eq.L}. 
Assume $w\in C^{\frac{\beta}{2s},\beta}_{t,x}((-\infty,0)\times\R^n)$, $a\in C^{1,\gamma}(S^{n-1})$ and $f\in C^{\frac{\gamma}{2s},\gamma}_{t,x}(Q_1)$. 
Then
\begin{equation}\label{eq.blow}
\sup_{r>0} r^{-\alpha-s}[w-(a_rt+x^TA_r x+p_r\cdot x+q_r)]_{C_{t,x}^{\frac{\beta}{2s},\beta}(Q_r)}\leq C [w]_{C_{t,x}^{\frac{\beta}{2s},\beta}((-\infty,0)\times\R^n)}
\end{equation}
where
\begin{equation}\label{eq.prblow}
a_rt+x^TA_r x+p_r\cdot x+q_r=\arg\min_{\mathcal{P}}\int_{-r^{2s}}^0\int_{B_r}(w-(at+x^TAx+p\cdot x+q))^2\:dxdt.
\end{equation}
where $\mathcal{P}$ is the space of all polynomials of degree at most $\lfloor (\alpha+\beta+s)/2s\rfloor$ in $t$ and $\lfloor \alpha+\beta+s\rfloor$ in $x$ (i.e., and $a_r=0$ if $\alpha+\beta<s$, $A_r=0$ if $\alpha+\beta+s<2$, and $p_r=0$ if $\alpha+\beta+s<1$).

In particular
\begin{equation}\label{eq.intest}
[w]_{C_{t,x}^{\frac{\alpha+\beta+s}{2s},\alpha+\beta+s}(Q_{1/2})}\leq C[w]_{C_{t,x}^{\frac{\beta}{2s},\beta}((-\infty,0)\times\R^n)}.
\end{equation}
\end{prop}

\begin{proof}
The proof is similar to that of Proposition \ref{prop.mainprop}. 
Assume \eqref{eq.blow} does not hold. 
Then there are sequences $w_k$, $L_k$ and $f_k$ satisfying:
\begin{itemize}
 \item $L_k$ is of the form \eqref{eq.L} and $a_k\in C^{1,\gamma}(S^{n-1})$
 \item $w_k\in C_{t,x}^{\beta/2s,\beta}((-\infty,0)\times\R^n)$ and $f_k\in C_{t,x}^{\gamma/2s,\gamma}(Q_1)$
 \item $w_k$ is a solution of
\begin{equation}\label{eq.eqwk}
 \partial_tw_k-L_kw_k=f_k\quad\textrm{ in } Q_1
\end{equation}
\end{itemize}
such that
\[
\sup_k\:\sup_{r>0}\:r^{-\alpha-s}[w_k-(a_{r,k}t+x^TA_{r,k}x+p_{r,k}\cdot x+q_{r,k})]_{C_{t,x}^{\beta/2s,\beta}(Q_r)}=\infty.
\]
with
\[
a_{r,k}t+x^TA_{r,k}x+p_{r,k}\cdot x+q_{r,k}=\arg\min_{\mathcal{P}}\int_{-r^{2s}}^0\int_{B_r}(w_k-(at+x^TAx+p\cdot x+q))^2\:dxdt.
\]
Define
\[
\theta(r):=\sup_k\:\sup_{r'>r} (r')^{-\alpha-s}[w_k-(a_{r',k}t+x^TA_{r',k}x+p_{r',k}\cdot x+q_{r',k})]_{C_{t,x}^{\frac{\beta}{2s},\beta}(Q_{r'})}.
\]
As before, $\theta$ is a nondecreasing function of $r$ that goes to $\infty$ as $r$ goes to 0. 
Moreover, since for any fixed $r$ we have $\theta(r)<\infty$, if we take the sequence $1/m$
there are $r'_m$ and $k_m$ such that $r'_m\geq1/m$ and
\[
(r_m')^{-\alpha-s}[w_{k_m}-(a_{r_m',k_m}t+x^TA_{r_m',k_m}x+p_{r_m',k_m}\cdot x+q_{r_m',k_m})]_{C_{t,x}^{\beta/2s,\beta}(Q_{r_m'})}\geq \frac{1}{2}\theta(1/m)\geq\frac{1}{2}\theta(r'_m).
\]
Let's denote
\[
w_m=w_{k_m} \quad A_m=A_{r_m',k_m} \quad p_m=p_{r_m',k_m} \quad q_m=q_{r_m',k_m} \quad a_m=a_{r_m',k_m}
\]
to make the notation cleaner, and further denote
\[
P_{r,m}(t,x)=a_mr^{2s}t+x^TA_mxr^2+p_m\cdot rx+q_m.
\] 

With this notation we define the blow-up sequence
\[
v_m(t,x)=\frac{w_m((r'_m)^{2s}t,r'_mx)-P_{r'_m,m}(t,x)}{(r'_m)^{\alpha+\beta+s}\theta(r'_m)}.
\]
Notice that
\begin{equation}\label{eq.boundvbelblow}
[v_m]_{C_{t,x}^{\beta/2s,\beta}(Q_1)}\geq 1/2
\end{equation}
and that
\begin{equation}\label{eq.ortblow}
\int_{-1}^0\int_{B_1}v_m(at+x^TAx+p\cdot x+q)\:dx\:dt=0
\end{equation}
for all $A\in\mathbb{R}^{n\times n}$, $p\in\mathbb{R}^n$ and $a,q\in\mathbb{R}$ because of the minimization condition \eqref{eq.prblow}. 

Next we show that 
\begin{equation}\label{eq.boundvblow}
[v_m]_{C^{\beta/2s,\beta}(Q_R)}\leq CR^{\alpha+s}
\end{equation}
for any $R\geq 1$. As before, this requires an estimate of the form 
\begin{equation}\label{eq.quad}
[P_{Rr'_m,m}-P_{r'_m,m}]_{C_{t,x}^{\beta/2s,\beta}(Q_{Rr'_m})}\leq \theta(r'_m)(Rr_m')^{\alpha+s}.
\end{equation}
The only difference between the proof \eqref{eq.quad} and \eqref{eq.linear} is the presence of a quadratic term so, to keep the presentation clear, let us assume for simplicity that the linear terms vanish and prove the bound for the quadratic term.

Hence, proceeding dyadically as in Proposition \ref{prop.mainprop}, we have to show that 
\[
[x^TA_{2^kr'_m,m}x-x^TA_{r'_m,m}x]_{C_{t,x}^{\beta/2s,\beta}(Q_{Rr'_m})}\leq C\theta(r'_m)(Rr_m')^{\alpha+s}.
\]

Let us start with $k=1$ and noticing that  
\begin{eqnarray*}
\frac{1}{(r'_m)^{\alpha+s}\theta(r'_m)}|x^T(A_{2r'_m,m}-A_{r'_m,m})x|(r'_m)^{2-\beta} & \leq & \frac{1}{(r'_m)^{\alpha+s}\theta(r'_m)}[x^TA_{2r'_m,m}x-x^TA_{r'_m,m}x]_{C_{t,x}^{\beta/2s,\beta}(Q_{r'_m})} \\
                                                                  & \leq & \frac{2^{\alpha+s}\theta(2r'_m)}{\theta(r'_m)}\frac{[w_m-x^TA_{2r'_m,m}x]_{C_{t,x}^{\beta/2s,\beta}(Q_{2r'_m})}}{(2r'_m)^{\alpha+s}\theta(2r'_m)} \\
								 &  & \frac{1}{(r'_m)^{\alpha+s}\theta(r'_m)}[w_m-x^TA_{r'_m,m}x]_{C_{t,x}^{\beta/2s,\beta}(Q_{r'_m})} \\
                                                                  & \leq & C
\end{eqnarray*}
so that, denoting by $\|\cdot\|$ is the $L^2$ matrix norm, i.e. $\|M\|=\sup_{|x|=1}|x^TMx|$ and taking supremum over $x$
\[
\|A_{2r'_m,m}-A_{r'_m,m}\|\leq C(r'_m)^{\alpha+\beta+s-2}\theta(r'_m).
\]
But on the other hand, 
\[
[x^TA_{2r'_m,m}x-x^TA_{r'_m,m}x]_{C_{t,x}^{\beta/2s,\beta}(Q_{2r'_m})}\leq C\|A_{2r'_m,m}-A_{r'_m,m}\|(r'_m)^{2-\beta}
\]
so the result follows in this case. For $R=2^k$ just use a telescopic sum as in Proposition \ref{prop.mainprop}.

Using \eqref{eq.quad} we have
\begin{eqnarray*}
[v_m]_{C_{t,x}^{\beta/2s,\beta}(Q_R)} & = & \frac{1}{(r'_m)^{\alpha+s}\theta(r'_m)}[w_m-P_{r'_m,m}]_{C_{t,x}^{\beta/2s,\beta}(Q_{Rr'_m})} \\
			& = & \frac{R^{\alpha+s}}{(Rr'_m)^{\alpha+s}\theta(r'_m)}[w_m-P_{r'_m,m}]_{C_{t,x}^{\beta/2s,\beta}(Q_{Rr'_m})} \\
                         & \leq & \frac{R^{\alpha+s}}{(Rr'_m)^{\alpha+s}\theta(r'_m)}[w_m-P_{Rr'_m,m}]_{C_{t,x}^{\beta/2s,\beta}(Q_{Rr'_m})}\\
                         &  & +\frac{R^{\alpha+s}}{(Rr'_m)^{\alpha+s}\theta(r'_m)}[P_{Rr'_m,m}-P_{r'_m,m}]_{C_{t,x}^{\beta/2s,\beta}(Q_{Rr'_m})} \\ 
                         & \leq & \frac{R^{\alpha+s}\theta(Rr'_m)}{\theta(r'_m)}+C\frac{R^{\alpha+s}}{(Rr'_m)^{\alpha+s}\theta(r'_m)}(Rr'_m)^{\alpha+s}\theta(r'_m)\\
			& \leq & CR^{\alpha+s}
\end{eqnarray*}
and we get \eqref{eq.boundvblow}. From \eqref{eq.boundvblow} the bound
\begin{equation}\label{eq.unifvblow}
\|v_m\|_{L^\infty{(Q_R)}}\leq CR^{\alpha+\beta+s}
\end{equation}
follows as in Proposition \ref{prop.mainprop}.

By the Arzel\`a-Ascoli theorem, \eqref{eq.boundvblow} and \eqref{eq.unifvblow} imply that a subsequence of $\{v_m\}_m$ converges uniformly on compact subsets of $(-\infty,0)\times\mathbb{R}^n$ to
a uniformly continuous function $v$. 
Let's check that $v$ satisfies the hypotheses of Lemma \ref{lem.liou}. 
As in Proposition \ref{prop.mainprop} we have
\[
(\partial_t-L_{k_m})(v_m(t+\tau,x+h)-v_m(t,x))\longrightarrow 0
\]
as $m\rightarrow\infty$ (notice that the quadratic term can only appear when $s>1/2$ and in this case the operator is well defined on linear functions and vanishes identically). So, again using Lemma \ref{lem.limwm}
\[
(\partial_t-L)(v(t+\tau,x+h)-v(t,x))=0
\]
for some $L$ of the form \eqref{eq.L}. 
Finally, because of \eqref{eq.boundvblow}, $v$ also satisfies the growth condition of Lemma \ref{lem.liou}. Hence
\[
v(t,x)=at+x^TAx+p\cdot x+q,
\]
and by \eqref{eq.ortblow} we obtain $v\equiv0$. 
But passing to the limit in \eqref{eq.boundvbelblow} we get a contradiction, so \eqref{eq.blow} holds. The fact that \eqref{eq.blow} implies \eqref{eq.intest} is quite standard.
\end{proof}

\begin{proof}[Proof of Proposition \ref{prop.interior}]
Fix $r>0$ and consider the following rescaling of $w$:
\[
 \tilde{w}(t,x)=\frac{1}{r^{\alpha+\beta+s}}w(r^{2s}t,rx).
\]
Then $\tilde{w}$ will satisfy
\[
 \|\tilde{w}\|_{L^\infty((-\infty,0)\times\R^n)}<\infty\quad\textrm{ and }  \|\tilde{w}\|_{L^\infty(Q_R)}\leq CR^{\alpha+\beta+s}.
\]
Let now $\eta$ be a cut-off function in space that vanishes outside $B_1(e_n)$. 
More precisely, let $\eta\in C^\infty_c(B_1(e_n))$ with $\eta\equiv 1$ in $B_{5/6}(e_n)$.
Notice that $\tilde{w}\eta$ satisfies an equation like the one in Proposition \ref{prop.blowup}. 
Indeed, because $a\in C^{1,\gamma}(S^{n-1})$ we have that,
\[
 \partial_t\tilde{w\eta}-L\tilde{w}\eta=\tilde{f}\quad\textrm{ in }\quad (-1,0)\times B_{3/4}(e_n)
\]
and hence \eqref{eq.intest} gives (recall $\tilde{w}=\tilde{w}\eta$ in $B_{1/2}$)
\[
[\tilde{w}]_{C_{t,x}^{\frac{\alpha+\beta+s}{2s},\alpha+\beta+s}((-1,0)\times B_{1/2}(e_n))}\leq C[\tilde w]_{C_{t,x}^{\frac{\beta}{2s},\beta}((-1,0)\times B_1(e_n))},
\]
and in particular also
\[
[\tilde{w}]_{C_{t,x}^{\frac{\alpha+\beta}{2s},\alpha+\beta}((-1,0)\times B_{1/2}(e_n))}\leq C[\tilde w]_{C_{t,x}^{\frac{\beta}{2s},\beta}((-1,0)\times B_1(e_n))}.
\]
Rescaling this estimate back we get
\[
[w]_{C_{t,x}^{\frac{\alpha+\beta+s}{2s},\alpha+\beta+s}((-r^{2s},0)\times B_{r/2}(e_n))}\leq r^{-\alpha-s}C[w]_{C_{t,x}^{\frac{\beta}{2s},\beta}((-1,0)\times B_r(e_n))}\leq Cr^{-\alpha-s}r^{s+\alpha}=C
\]
and
\[
[w]_{C_{t,x}^{\frac{\alpha+\beta}{2s},\alpha+\beta}((-r^{2s},0)\times B_{r/2}(e_n))}\leq r^{-\alpha}C[w]_{C_{t,x}^{\frac{\beta}{2s},\beta}((-1,0)\times B_r(e_n))}\leq Cr^{-\alpha}r^{s+\alpha}=Cr^s,
\]
as wanted.
\end{proof}


\begin{thebibliography}{00}

\bibitem[AV17]{AV17} N. Abatangelo, E. Valdinoci, \textit{Getting acquainted with the fractional Laplacian}, preprint arXiv (2017).

\bibitem[Ap04]{Ap04} D. Applebaum, \textit{L\'evy processes --- From Probability to Finance and Quantum Groups}, Notices AMS 51 (2004), 1336-1347.

\bibitem[BFR17]{BFR2} B. Barrios, A. Figalli, X. Ros-Oton, \emph{Free boundary regularity in the parabolic fractional obstacle problem}, Comm. Pure Appl. Math., in press (2018).

\bibitem[Bo97]{Bo97} K. Bogdan, \textit{The boundary Harnack principle for the fractional Laplacian}, Studia Math., \textbf{123} (1997), 43-80.

\bibitem[BKK08]{BKK08} K. Bogdan, T. Kulczycki, and M. Kwasnicki, \textit{Estimates and structure of $\alpha$-harmonic functions}, Probab. Theory Related Fields \textbf{140} (2008), 345-381.

\bibitem[BKK15]{BKK15} K. Bogdan, T. Kumagai, M. Kwasnicki, \textit{Boundary Harnack inequality for Markov processes with jumps}, Trans. Amer. Math. Soc. \textbf{367} (2015), 477-517.

\bibitem[BV16]{BV16} C. Bucur, E. Valdinoci, \textit{Nonlocal Diffusions and Applications}, Lecture Notes of the Unione Matematica Italiana, Springer 2016.

\bibitem[CC95]{CC} L. Caffarelli, X. Cabr\'e, \emph{Fully Nonlinear Elliptic Equations}, AMS Colloquium Publications, Vol 43, 1995.

\bibitem[CD14]{CD14} H. Chang-Lara, G. Dávila, \emph{Regularity for solutions of nonlocal parabolic equations II}, J. Differential Equations 256 (2014), 130-156.

\bibitem[CK15]{CK15} H. Chang-Lara, D. Kriventsov, \emph{Further time regularity for non-local, fully non-linear parabolic equations}, Comm. Pure Appl. Math. 70 (2017), 950-977.

\bibitem[CKS10]{CKS10} Z. Chen, P. Kim, R. Song, \emph{Heat kernel estimates for the Dirichlet fractional Laplacian}, J. Eur. Math. Soc. \textbf{12} (2010), 1307-1329.

\bibitem[CT16]{CT16} R. Cont, P. Tankov, \emph{Financial Modelling With Jump Processes}, Financial Mathematics Series. Chapman \& Hall/CRC, Boca Raton, FL, 2004.

\bibitem[DS16]{DS16} D. De Silva, O. Savin, \textit{Boundary Harnack estimates in slit domains and applications to thin free boundary problems}, Rev. Mat. Iberoam. 32 (2016), 891-912.

\bibitem[FR16]{FR1} X. Fern\'andez-Real, X. Ros-Oton, \emph{Regularity theory for general stable operators: parabolic equations}, J. Funct. Anal, 272 (2017), 4165-4221.

\bibitem[Ga17]{Ga17} N. Garofalo, \emph{Fractional thoughts}, preprint arXiv (2017).

\bibitem[Go16]{Go16} M. d. M. Gonzalez, \emph{Recent progress on the fractional Laplacian in conformal geometry}, Chapter in the book ``Recent Developments in the Nonlocal Theory'', De Gruyter 2018.

\bibitem[Gr15]{GG15} G. Grubb, \emph{Fractional Laplacians on domains, a development of H\"ormander's theory of $\mu$-transmission pseudodifferential operators}, Adv. Math. 268 (2015), 478-528.

\bibitem[Gr14]{GG14} G. Grubb, \emph{Local and nonlocal boundary conditions for $\mu$-transmission and fractional elliptic pseudodifferential operators}, Anal. PDE 7 (2014), 1649-1682.

\bibitem[Gr17]{GG17} G. Grubb, \emph{Regularity in $L_p$ Sobolev spaces of solutions to fractional heat equations}, J. Funct. Anal., to appear (2018).

\bibitem[Gr18]{GG18} G. Grubb, \emph{Fractional-order operators: Boundary problems, heat equations}, to appear in Springer Proceedings in Mathematics and Statistics: ``New Perspectives in Mathematical Analysis - Plenary Lectures, ISAAC 2017, Vaxjo Sweden''.

\bibitem[JX15]{JX15} T. Jin, J. Xiong, \emph{Schauder estimates for solutions of linear parabolic integro-differential equations}, Discrete Contin. Dyn. Syst. A. 35 (2015), 5977-5998.

\bibitem[La72]{L72} N. S. Landkof, \emph{Foundations of Modern Potential Theory}, Springer, New York, 1972.

\bibitem[MK00]{MK00} R. Metzler, J. Klafter, \emph{The random walk's guide to anomalous diffusion: a fractional dynamics approach}, Phys. Rep. 339 (2000), 1-77.

\bibitem[Ro16]{Ro16} X. Ros-Oton, \emph{Nonlocal elliptic equations in bounded domains: a survey}, Publ. Mat. 60 (2016), 3-26.

\bibitem[Ro17]{R1} X. Ros-Oton, \emph{Boundary regularity, Pohozaev identities, and nonexistence results}, Chapter in the book ``Recent Developments in the Nonlocal Theory'', De Gruyter 2018.

\bibitem[RS16]{RS1} X. Ros-Oton, J. Serra, \emph{Regularity theory for general stable operators}, J. Differential Equations 260 (2016), 8675-8715.

\bibitem[RS16b]{RS2} X. Ros-Oton, J. Serra, \emph{Boundary regularity for fully nonlinear integro-differential equations}, Duke Math. J. 165 (2016), 2079-2154.

\bibitem[ST94]{ST94} G. Samorodnitsky, M. S. Taqqu, \emph{Stable Non-Gaussian Random Processes: Stochastic Models With Infinite Variance}, Chapman and Hall, New York, 1994.

\bibitem[SS16]{SS16} R. W. Schwab, L. Silvestre, \emph{Regularity for parabolic integro-differential equations with very irregular kernels}, Anal. PDE \textbf{9} (2016), 727-772.

\bibitem[Se15]{Se15} J. Serra, \emph{Regularity for fully nonlinear nonlocal parabolic equations with rough kernels}, Calc. Var. Partial Differential Equations 54 (2015), 615-629.

\bibitem[SW99]{SW99} R. Song, J.-M. Wu, \textit{Boundary Harnack principle for symmetric stable processes}, J. Funct. Anal. \textbf{168} (1999), 403-427.

\end{thebibliography}
\end{document}